\newcommand{\R}{\mathds R}

\documentclass[twoside,a4paper,11pt]{amsart}

\usepackage{amsmath}
\usepackage{times}
\usepackage{hyperref}
\usepackage{dsfont}
\usepackage{graphicx}
\usepackage{epsfig}

\numberwithin{equation}{section}

\title[Examples with minimal number of brake orbits and homoclinics]{Examples with minimal number of brake orbits and homoclinics in
annular potential regions}

\author[R. Giamb\`o]{Roberto Giamb\`o}
\author[F. Giannoni]{Fabio Giannoni}
\address{Scuola di Scienze e Tecnologie, Sezione di Matematica\hfill\break\indent
Universit\`a di Camerino \hfill\break\indent
 I-62032 Camerino (MC), Italy}
\email{roberto.giambo@unicam.it, fabio.giannoni@unicam.it}
\author[P. Piccione]{Paolo Piccione}
\address{Departamento de Matem\'atica \hfill\break\indent Universidade de S\~ao Paulo \hfill\break\indent
Rua do Mat\~ao, 1010 \hfill\break\indent 05508-090 S\~ao Paulo, SP, Brazil}
\email{piccione.p@gmail.com}

\urladdr{http://www.ime.usp.br/\~{}piccione}
\date{September 18th, 2012}

\subjclass[2000]{37C29, 37J45, 58E10}

\begin{document}


\theoremstyle{plain}\newtheorem{teo}{Theorem}[section]
\theoremstyle{plain}\newtheorem{prop}[teo]{Proposition}
\theoremstyle{plain}\newtheorem{lem}[teo]{Lemma}
\theoremstyle{plain}\newtheorem{cor}[teo]{Corollary}
\theoremstyle{definition}\newtheorem{defin}[teo]{Definition}
\theoremstyle{remark}\newtheorem{rem}[teo]{Remark}
\theoremstyle{definition}\newtheorem{example}[teo]{Example}

\theoremstyle{plain}\newtheorem*{teon}{Theorem}


\begin{abstract}
We use a geometric construction to exhibit examples of autonomous Lagrangian
systems admitting exactly two homoclinics emanating from a nondegenerate
maximum of the potential energy
and reaching  a regular level of the potential having the same value of the maximum point. 
Similarly, we show examples of Hamiltonian
systems that admit exactly two brake orbits in an annular potential region connecting the two connected components of the boundary of the potential well.
These examples show that the estimates proven in \cite{arma} are sharp.
\end{abstract}

\maketitle

\renewcommand{\contentsline}[4]{\csname nuova#1\endcsname{#2}{#3}{#4}}
\newcommand{\nuovasection}[3]{\medskip\hbox to \hsize{\vbox{\advance\hsize by -1cm\baselineskip=12pt\parfillskip=0pt\leftskip=3.5cm\noindent\hskip -2cm #1\leaders\hbox{.}\hfil\hfil\par}$\,$#2\hfil}}
\newcommand{\nuovasubsection}[3]{\medskip\hbox to \hsize{\vbox{\advance\hsize by -1cm\baselineskip=12pt\parfillskip=0pt\leftskip=4cm\noindent\hskip -2cm #1\leaders\hbox{.}\hfil\hfil\par}$\,$#2\hfil}}

\tableofcontents

\section{Introduction}
\label{sec:intro}
Let $(M,g)$ be a complete Riemannian manifold, $\Omega\subset M$
an open subset whose closure is homeomorphic to an annulus.
In \cite{arma} it was proved
that if $\partial \Omega$ is smooth and it satisfies a
strong concavity assumption, then there are at least two distinct
geodesics  in $\overline\Omega=\Omega \bigcup\partial\Omega$ starting orthogonally to one
connected component of $\partial\Omega$ and arriving orthogonally onto
the other one. Using Maupertuis' principle and the results in \cite{GGP1}, one
obtains a proof of the existence  of at least two distinct homoclinic orbits for an autonomous
Lagrangian system emanating from a  nondegenerate maximum point of
the potential energy and reaching  a regular level of the potential having the same value of the maximum point. Moreover one obtain also a proof of the existence of at least two distinct
{\em brake orbits\/} for a class of Hamiltonian systems, connecting the two connected components of the boundary of the potential well. 
In this paper we show, by construction of examples, that this kind of estimate cannot be improved.

Recall that Maupertuis' principle relates solutions of dynamical systems at a fixed energy
level, with geodesics in metrics that are conformal to the kinetic energy (Jacobi metrics),
see \cite{GGP1} for details.
It is easy to produce examples of \emph{spherical annuli}, i.e., Riemannian annuli with strongly concave boundary,
that admit only two orthogonal geodesic chords, see \cite{arma}.
However, in principle the mere existence of spherical annuli with only two orthogonal geodesic chords
does not provide examples of Lagrangian or Hamiltonian systems as above with only two
homoclinics or brake orbits at a given energy level. Namely, such examples do not necessarily
arise from dynamical systems via Maupertuis' principle. The point addressed in the
present paper is to describe a sort of inverse process of Maupertuis' principle, in other
words, how to go from orthogonal geodesics to brake orbits. Given a spherical annulus,
we construct a $C^2$-potential function $V$ in an open subset of $\mathbb{R}^N$, and an energy level $E$ which is
non critical for $V$, such that the corresponding brake orbits of energy $E$ and connecting the two different connecteed components of the boundary of the potential well are into one-to-one
correspondence with the orthogonal geodesic chords in the spherical annulus.

It is an interesting observation that the construction of such a potential
$V$ relies on an abstract characterization of geodesics in conformal metrics
which is proved in Section~\ref{sec:costruzioneGeo}. Such a characterization, which has an
independent interest in its own, is given in terms of a sort of converse of the classical Gauss Lemma
in Riemannian geometry. Roughly speaking, a family of curves starting orthogonally to a given hypersurface of a Riemannian manifold
are geodesics in some conformal metric if and only if Gauss Lemma holds for the family, see Proposition~\ref{thm:recGaussLemma} for a precise statement.

In Theorem \ref{thm:deuBO} it is got the existence of potentials having only two brake orbits connecting the two connected components of the potential well. the analogous results for homoclinics is given in Theorem \ref{thm:deuHOM}.

The problem of the multiplicity  for brake orbits and homoclinics is  widely studied. We refer for example to the papers \cite{arma}--\cite{ZhangLiu} and reference therein, for an updated list of results in this field.

The paper is organized as follows. In section \ref{sec:gauss} we give a result concerning constructions of
potentials with prescribed geodesics for the related conformal metric. 
In section \ref{sec:OGCBO} we give the concrete construction of the potentials extending a given potential in the unit disk. Finally,
in section \ref{sec:examples} we use the results of section \ref{sec:OGCBO} to show that the estimates of th multiplicity results proved in \cite{arma} cannot be improved.

\section{Prescribing orthogonal geodesics in conformal metrics}\label{sec:gauss}
\label{sec:costruzioneGeo}
Gauss Lemma in Riemannian Geometry says that the exponential map, although not an isometry, preserves
orthogonality with respect to radial geodesics. In this section we prove a converse of this statement for
conformal metrics: given a family of curves in a Riemannian manifold, issuing orthogonally from a given hypersurface,
then they are unit speed orthogonal geodesics with respect to some conformal metric if and only if Gauss Lemma
holds.
\smallskip

Let $(M,g)$ be a Riemannian manifold, $\Sigma\subset M$ a hypersurface of $M$ and let
$y:[a,b]\times\Sigma\to M$ be a $\mathcal C^k$-diffeomorphism onto a domain $D\subset M$,
with $k\ge2$ and $0\in\left]a,b\right[$, satisfying:
\begin{align}
&y(0,p)=p,\quad\forall\,p\in\Sigma;\label{eq:itema}
\\
&\frac{\partial y}{\partial\tau}(p,0)\in T_p\Sigma^\perp,\quad\forall\,p\in\Sigma.\label{eq:itemb}
\end{align}
\begin{prop}\label{thm:recGaussLemma}
There exists a positive $\mathcal C^{k-1}$-function $\Psi:D\to\mathds R^+$ such that, for all $p\in\Sigma$,
the curve $[a,b]\ni\tau\mapsto y(\tau,p)$ is a unit speed geodesic with respect to the conformal metric $\widetilde g=\Psi\cdot g$
in $D$ if and only if:
\begin{equation}\label{eq:gausslemma}
g\big(\tfrac{\partial y}{\partial\tau}(\tau,p),\mathrm dy(\tau,p)[w]\big)=0,
\end{equation}
for all $(\tau,p)\in[a,b]\times\Sigma$ and all $w\in T_p\Sigma$. In this case, the conformal factor $\Psi$ is given by:
\begin{equation}\label{eq:conffactor}
\Psi\big(y(\tau,p)\big)=g\Big(\tfrac{\partial y}{\partial\tau}(\tau,p),\tfrac{\partial y}{\partial\tau}(\tau,p)\Big)^{-1}.
\end{equation}
\end{prop}
\begin{proof}
If the curves $\tau\mapsto y(\tau,p)$ are geodesics with respect to some conformal metric $\widetilde g$, then
equality \eqref{eq:gausslemma} follows immediately from Gauss Lemma. Namely, the map $y$ is the $\widetilde g$-exponential map
of $T\Sigma^\perp$, and therefore, for $w\in T_p\Sigma^\perp$:
\[\widetilde g\big(\tfrac{\partial y}{\partial\tau}(\tau,p),\mathrm dy(\tau,p)[w]\big)=\widetilde g\big(\tfrac{\partial y}{\partial\tau}(0,p),\mathrm dy(0,p)[w]\big)
\stackrel{\text{by \eqref{eq:itema}}}{=}\widetilde g\big(\tfrac{\partial y}{\partial\tau}(0,p),w\big)\stackrel{\text{by \eqref{eq:itemb}}}{=}0.\]
For the converse, we first observe that the metric $\widetilde g=\Psi\cdot g$, with $\Psi$ given in \eqref{eq:conffactor}, is the only metric conformal
to $g$ for which the curves $\tau\mapsto y(\tau,p)$ have unit speed. In particular, once the first statement of the Proposition is proved, the
last statement is trivial.  In order to conclude the proof, it suffices to show that, assuming \eqref{eq:gausslemma}, then the curves $\tau\mapsto y(\tau,p)$
are geodesics relatively to the metric $\widetilde g=\Psi\cdot g$, with $\Psi$ given in \eqref{eq:conffactor}.
An elementary calculation show that $\gamma:[a,b]\to D$ is a $\widetilde g$-geodesic iff it satisfies:
\begin{equation}\label{eq:geoconfmetric}
\frac{\mathrm D}{\mathrm d\tau}\Big[\Psi(\gamma)\,\gamma'\Big]=\tfrac12\,g(\gamma',\gamma')\,\nabla\Psi(\gamma),
\end{equation}
where $\frac{\mathrm D}{\mathrm d\tau}$ denotes $g$-covariant differentiation along $\gamma$, $\gamma'$ is the tangent field along $\gamma$,
and $\nabla\Psi$ is the $g$-gradient of $\Psi$.
Set $\gamma(\tau)=y(\tau,p)$; $g$-contraction of both sides of \eqref{eq:geoconfmetric} with $\gamma'=\tfrac{\partial y}{\partial\tau}$ gives an equality, as a consequence of the fact
 that $\Psi(\gamma)g(\gamma',\gamma')$ is constant along $\gamma$. Since $\mathrm dy$ is onto, in order to conclude it suffices to show that
 \begin{equation}\label{eq:geoconfmetric2}
g\left(\tfrac{\mathrm D}{\mathrm d\tau}\big[\Psi(y)\,\tfrac{\partial y}{\partial\tau}\big],\mathrm dy[w]\right)=\tfrac12\,g\big(\tfrac{\partial y}{\partial\tau},\tfrac{\partial y}{\partial\tau}\big)\,g\big(\nabla\Psi(y),\mathrm dy[w]\big),
\end{equation}
for all $w\in T_p\Sigma^\perp$.
Straightforward calculations give:
\begin{multline*}
g\big(\nabla\Psi(y),\mathrm dy[w]\big)=-2\,g\big(\tfrac{\partial y}{\partial\tau},\tfrac{\partial y}{\partial\tau}\big)^{-2}g\big(\nabla_{\mathrm dy[w]}(\tfrac{\partial y}{\partial\tau}),\tfrac{\partial y}{\partial\tau}\big)\\=-2\,g\big(\tfrac{\partial y}{\partial\tau},\tfrac{\partial y}{\partial\tau}\big)^{-2}g\big(\tfrac{\mathrm D}{\mathrm d\tau}\mathrm dy[w],\tfrac{\partial y}{\partial\tau}\big)\\=
-2\,g\big(\tfrac{\partial y}{\partial\tau},\tfrac{\partial y}{\partial\tau}\big)^{-2}\,\tfrac{\mathrm d}{\mathrm d\tau}\,g\big(\mathrm dy[w],\tfrac{\partial y}{\partial\tau}\big)+
2\,g\big(\tfrac{\partial y}{\partial\tau},\tfrac{\partial y}{\partial\tau}\big)^{-2}g\big(\mathrm dy[w],\tfrac{\mathrm D}{\mathrm d\tau}\tfrac{\partial y}{\partial\tau}\big)
\\ \stackrel{\text{by \eqref{eq:gausslemma}}}{=}2\,g\big(\tfrac{\partial y}{\partial\tau},\tfrac{\partial y}{\partial\tau}\big)^{-2}g\big(\mathrm dy[w],\tfrac{\mathrm D}{\mathrm d\tau}\tfrac{\partial y}{\partial\tau}\big),
\end{multline*}
thus the right-hand side of \eqref{eq:geoconfmetric2} is given by:
\begin{equation}\label{eq:RD}
\tfrac12\,g\big(\tfrac{\partial y}{\partial\tau},\tfrac{\partial y}{\partial\tau}\big)\,g\big(\nabla\Psi(y),\mathrm dy[w]\big)=g\big(\tfrac{\partial y}{\partial\tau},\tfrac{\partial y}{\partial\tau}\big)^{-1}
g\big(\mathrm dy[w],\tfrac{\mathrm D}{\mathrm d\tau}\tfrac{\partial y}{\partial\tau}\big).
\end{equation}
The left-hand side of \eqref{eq:geoconfmetric2} is also computed easily as:
\begin{multline}\label{eq:LD}
g\left(\tfrac{\mathrm D}{\mathrm d\tau}\big[\Psi(y)\,\tfrac{\partial y}{\partial\tau}\big],\mathrm dy[w]\right)\\=g\big(\nabla\Psi(y),\tfrac{\partial y}{\partial\tau}\big)\,g\big(\tfrac{\partial y}{\partial\tau},\mathrm dy[w]\big)+\Psi(y)\,
g\big(\mathrm dy[w],\tfrac{\mathrm D}{\mathrm d\tau}\tfrac{\partial y}{\partial\tau}\big)\\
\stackrel{\text{by \eqref{eq:gausslemma}}}{=}g\big(\tfrac{\partial y}{\partial\tau},\tfrac{\partial y}{\partial\tau}\big)^{-1}
g\big(\mathrm dy[w],\tfrac{\mathrm D}{\mathrm d\tau}\tfrac{\partial y}{\partial\tau}\big).
\end{multline}
Equality \eqref{eq:geoconfmetric2} follows readily from \eqref{eq:RD} and \eqref{eq:LD}. The proof is completed.
\end{proof}

\section{From Orthogonal Geodesic Chords to Brake Orbits}\label{sec:OGCBO}
\label{sec:C2counterexample}
In this section we will establish a one-to-one correspondence between Jacobi geodesics arriving on the boundary
of the potential well and orthogonal geodesic chords in a conformally flat metric.

\subsection{Family of geodesics orthogonal to the sphere}
We will first define families of curves orthogonal to the unit sphere $\mathbb S^{N-1}$, seen as a hypersurface of $\mathds R^{N}$,
satisfying the assumptions of Proposition~\ref{thm:recGaussLemma}. This will be done using maximal slope curves for a given function, as follows.
The method that we will use can be extended to very general situations.
Let us denote by $\mathrm{grad}$ the gradient with respect to the Euclidean structure.

\begin{prop}\label{costruzione-geo}
Let $F : \mathds{R}^{N}\setminus \{0\}\rightarrow \mathds{R}$ and $\phi: [a,b]\rightarrow \mathds{R}^+\setminus \{0\}$ be
maps of class $C^{\infty}$, and $-\infty < a < 0 < b < +\infty$. Assume  that
\begin{equation}\label{eq:ipotesi-su-F}
\begin{split}
F(P)=0 \text{ for any }P\in \mathbb{S}^{N-1},  \\[.3cm] \inf_{x \in F^{-1}([\alpha,\beta])}\Vert \text{grad }F(x)\Vert > 0,
\end{split}
\end{equation}
where $\alpha < 0 < \beta$ are defined by
\[
\alpha=\int_0^a\phi(\theta)\,\mathrm d\theta, \;\beta=\int_0^b\phi(\theta)\,\mathrm d\theta.
\]
For all $P\in\mathbb S^{N-1}$, let  $\tau\mapsto y(P,\tau)$ be the local solution of the Cauchy problem:
\begin{equation}\label{eq:FCauchy1}
\left\{
\begin{aligned}
&\frac{\partial y}{\partial \tau}(P,\tau)= \phi(\tau)\frac{\text{ grad }F(y(P,\tau))}{\Vert \text{ grad }F(y(P,\tau)) \Vert^2}\\[.3cm]
& y(P,0)=P \in \mathbb{S}^{N-1}.
\end{aligned}
\right.
\end{equation}
Then $y(P,\tau)$ is well defined for any $P \in \mathbb{S}^{N-1}, \tau \in [a,b]$, and it satisfies \eqref{eq:itema},
\eqref{eq:itemb}  and \eqref{eq:gausslemma} with $g=g_0$ euclidean metric. Moreover $y:\mathbb{S}^{N-1} \times [a,b] \rightarrow F^{-1}([a,b])$ is a $C^{\infty}$--diffeomorphism.
\end{prop}
\begin{proof}
First note that by \eqref{eq:FCauchy1},
\[
\frac{\partial }{\partial \tau}\,F\big(y(P,\tau)\big) = \phi(\tau) \text{ for any }P\in \mathbb{S}^{N-1}, \; \tau \in [a,b].
\]
Since $F(y(P,0))\equiv 0$, we have
\begin{equation}\label{eq:Fsuy}
F\big(y(P,\tau)\big)=\int_0^{\tau}\phi(\theta)\,\mathrm d\theta \text{ for any }P \in \mathbb{S}^{N-1}, \; \tau \in [a,b].
\end{equation}
This shows that $y$ takes values in $F^{-1}\big([\alpha,\beta]\big)$, and thus, by  \eqref{eq:ipotesi-su-F} , $y(P,\tau)$ is well defined for  any $P \in \mathbb{S}^{N-1}, \tau \in [a,b]$. Moreover since $F$ and $\phi$ are of class $C^{\infty}$, also $y$ is.
Formula \eqref{eq:Fsuy}  also proves that $F\big(y(P,\tau)\big)$ is constant on $\mathbb S^{N-1}\times\{\tau\}$ for all $\tau$.
Using the fact that $\text{grad }F\big(y(P,\tau)\big)$ is parallel to $\frac{\partial y}{\partial \tau}(P,\tau)$, this implies easily \eqref{eq:gausslemma}.

Let us now show that $\mathrm dy$ is everywhere non singular. Given $P\in\mathbb S^{N-1}$, let $e_1,\ldots,e_{N-1}$ be
a basis of $T_P\mathbb S^{N-1}$, and set:
\[
v_i(P,\tau) = \mathrm dy(P,\tau)[e_i], \quad i=1,\ldots,{N-1}.
\]
Differentiating \eqref{eq:FCauchy1}, one obtains that $v_i$ satisfies the linear Cauchy problem
\begin{equation}\label{eq:FCauchyLin}
\left\{
\begin{aligned}
&\frac{\partial v_i}{\partial \tau}= \phi(\tau)\,\mathrm dG\big(y(P,\tau)\big)[v_i]\\[.3cm]
& v_i(0)=e_i,
\end{aligned}
\right.
\end{equation}
where $G$ is the map $G(y)=\frac{\text{ grad }F(y)}{\Vert \text{ grad }F(y)) \Vert^2}$.
Since the  $v_i$'s are linearly independent at $\tau=0$, they are everywhere linearly independent.
Moreover, by  \eqref{eq:gausslemma} , $\frac {\partial y}{\partial \tau}$ is always orthogonal to the $v_i$'s ,
and this implies that $\mathrm dy$ is everywhere nonsingular.

Moreover \eqref{eq:itema} is just the initial condition of the Cauchy problem \eqref{eq:FCauchy1}, while \eqref{eq:itemb} follows immediately
from the fact that, for any $P \in \mathbb{S}^{N-1}$, $\frac{\partial y}{\partial \tau}(P,0)$ is parallel to $\text{grad }F(P)$, which is orthogonal to $\mathbb{S}^{N-1}=F^{-1}(0)$.

It remains to prove that
$y:\mathbb{S}^{N-1} \times [a,b] \rightarrow F^{-1}([a,b])$ is a homeomorphism.
For the injectivity, assume $y(P_1,\tau_1)=y(P_2,\tau_2)$. Since $\phi$ is positive, the map $\tau \mapsto \int_0^{\tau}\phi(\theta)\,\mathrm d\theta$ is strictly increasing, therefore by  \eqref{eq:Fsuy},  we deduce $\tau_1=\tau_2$. Then, by the local uniqueness of the Cauchy problem
\eqref{eq:FCauchy1} we get also $P_1=P_2$.

To prove the surjectivity, fix $y_* \in F^{-1}([\alpha,\beta])$ and solve the Cauchy problem
\begin{equation}\label{eq:FCauchy2b}
\left\{
\begin{aligned}
&Y'= -\phi(\tau)\frac{\text{ grad }F(Y)}{\Vert \text{ grad }F(Y) \Vert^2}\\
& Y(0)=y_*
\end{aligned}
\right.
\end{equation}
and denote by $Y(y_*,\tau)$ its solution. Let $\tau_*$ be such that
\[\int_0^{\tau_*}\phi(\theta)\,\mathrm d\theta=F(y_*)\]
and take $P_*=Y(y_*,\tau_*)$, which is in $F^{-1}(0)=\mathbb{S}^{N-1}$. Then, $y(P_*,\tau_*)=y_*$, i.e. $y$ is surjective.
Finally, to prove the continuity of the inverse map, take a sequence $y_n$ in $F^{-1}([\alpha,\beta])$ such that $y_n \rightarrow y_0$.
Now let $(P_n,\tau_n)$ such that $y(P_n,\tau_n)=y_n$. Up to subsequences, we can assume that $(P_n,\tau_n)$
converges to $(P_0,\tau_0)$. But $y(P_n,\tau_n)$ converges to
$y(P_0,\tau_0)=y_0$,
and by injectivity there exists only one pair $(P_0,\tau_0)$ satisfying $y(P,\tau)=y_0$. This implies that every subsequence of $(P_n,\tau_n)$ converges
to $(P_0,\tau_0)$, which gives the continuity of the inverse map.
\end{proof}

\subsection{Maupertuis' Principle}
\label{sub:Maupertuis}
For the convenience of the reader, we will now give a statement of the classical Maupertuis principle in our framework.
Given $A\subset\mathds R^{N}$, a $C^2$-function $W:A\to\mathds R$, and a number $E > \sup_AW$,
we will denote by $g_{W,E}$ the conformally flat metric
\begin{equation}\label{eq:Jacobi}
g_{W,E}=\tfrac12\big(E-W(x)\big)g_0.
\end{equation}
This will be called the \emph{Jacobi metric with potential $E-W$}. The following result is the classical Maupertuis principle in our setting.
\begin{lem}\label{cambio-variabile}
Let $y:\mathbb S^{N-1}\times[a,b]\to D\subset\mathds R^{N}$ be a diffeomorphism of class $C^2$ onto a domain
$D\subset\mathds R^{N}$, satisfying  \eqref{eq:gausslemma}. Assume that $s\mapsto y(P,s)$ is a unit speed geodesic in $D$
for the metric \eqref{eq:Jacobi}, with potential $W:D\to\mathds R$ defined by $W=\mathcal W\circ y^{-1}$, where
$\mathcal W:\mathbb S^{N-1}\times[a,b]$ is:
\begin{equation}\label{eq:W}
\mathcal W(P,s)=E-{2}{\left\langle \frac{\partial y}{\partial s},\frac{\partial y}{\partial s}\right\rangle}^{-1}.
\end{equation}
Take
\begin{equation}\label{eq:ripJacobi}
t(P,s)=\int_a^s\frac{dr}{E-\mathcal W(P,r)}
\end{equation}
and denote by $\sigma(t)=\sigma(P,t)$ the inverse map of $t(P,\cdot)$, so that $\sigma$ solves the Cauchy problem
\begin{equation}\label{eq:sigmacauchyproblem}
\frac{d\sigma}{dt}=\left({\frac{dt}{ds}}\right)^{-1}=E-W\big(y(P,\sigma(t))\big),\quad \sigma(0)=0.
\end{equation}
Then, for any fixed $P$, the curve
\[
z(P,t)=y(P,\sigma(t))
\]
solves the dynamic equation
\begin{equation}\label{eq:Wdynamic-equation}
\ddot z(t) + \text{ grad }W(z(t))= 0,
\end{equation}
and the energy conservation law:
\begin{equation}\label{eq:consene}
\frac12\left\langle \frac{\partial z}{\partial t}(P,t),\frac{\partial z}{\partial t}(P,t)\right\rangle +W(z(P,t))\equiv E.
\qed\end{equation}
\end{lem}

\subsection{Construction of a global potential}
Now, we  choose $v \in \mathbb{S}^{N-1}$  and consider the potential
\begin{equation}\label{eq:uepsilon}
U_\epsilon(x)=\frac12\Vert x-\epsilon v\Vert^2,
\end{equation}
defined in the unit disk of $\mathds{R}^{N}$, where $\epsilon > 0$ will be fixed later.

 The construction of the $C^2$--potentials to obtain our examples in section \ref{sec:examples} is based on the idea of fixing fix $E > \frac12$ and extending the potential $U_\epsilon$ to a suitable potential $V_\epsilon$ outside the unit disk, so that $V_\epsilon^{-1}\big(\left]-\infty,E\right]\big)$ is homeomorphic to $B[0;1]$ and $E$ is a regular value of $V_\epsilon$. (The application to the topology of the annulus will be made in next section). The essential property of $V_\epsilon$ will be that all the solutions of the Cauchy problem
\begin{equation}\label{eq:startingBO}
\left\{
\begin{aligned}
& \ddot q + \text{grad }V_\epsilon(q) = 0 \\
& q(0) \in V_{\epsilon}^{-1}(E)\\
& \dot q(0)=0
\end{aligned}
\right.
\end{equation}
arrive orthogonally to $\mathbb{S}^{N-1}$.

First of all, we will use Proposition~\ref{costruzione-geo} to construct a family $y_\epsilon(P,\tau)$ of curves solving the Cauchy problem~\eqref{eq:FCauchy1}, where
 $\phi(\tau)={(1-\tau)^{-\frac13}}$ (this choice is suggested by the behavior of the Jacobi geodesics near the regular boundary of any potential well), and $F$
 is taken to be the function $F_\epsilon$ defined by:
\begin{equation}\label{F-epsilon}
F_{\epsilon}(x) = \sum_{i=1}^3\frac1{i!}f_{i,\epsilon}\left(\frac x{\Vert x\Vert}\right)\,\big(\Vert x\Vert-1\big)^i;
\end{equation}
note that this function vanishes on $\mathbb{S}^{N-1}$. Thus, for all $P\in\mathbb S^{N-1}$, the map $\tau\mapsto y_\epsilon(P,\tau)$ is a solution of the Cauchy problem:
\begin{equation}\label{eq:FCauchy2}
\left\{
\begin{aligned}
&\frac{\partial y}{\partial \tau}= {(1-\tau)^{-\frac13}}\frac{\text{ grad }F_\epsilon(y)}{\Vert \text{ grad }F_\epsilon(y) \Vert^2}\\[.3cm]
& y(0)=P.
\end{aligned}
\right.
\end{equation}
The $C^{\infty}$-maps $f_{i,\epsilon}$, ($i=1,2,3$) that appear in the definition of  $F_\epsilon$ in \eqref{F-epsilon} will be chosen to
obtain a $C^2$-matching on $\mathbb{S}^{N-1}$ between $U_\epsilon$ and the potential
\begin{equation}\label{W-epsilon}
\mathcal W_\epsilon(P,\tau)=E-2(1-\tau)^{\frac23}\langle \text{grad }F_\epsilon(y_\epsilon(P,\tau)),\text{grad }F_\epsilon(y_\epsilon(P,\tau))\rangle,
\end{equation}
coming from \eqref{eq:W} (with $y=y_\epsilon$).
\begin{prop}\label{th:F+V}
Fix $E \geq 4$. There exist $\epsilon_* \in\left]0,1\right]$, a bounded open subset $\mathcal O$ containing $B[0;1]$, and $f_{i,\epsilon} : \mathbb{S}^{N-1} \rightarrow \mathds{R}^+$ of class $C^\infty$, ($i=1,2,3$), such that, for any $\epsilon \in\left]0,\epsilon_*\right]$ the following properties are satisfied:

\begin{enumerate}
\item\label{en:1} $y_\epsilon(P,\tau)$ is defined with $C^\infty$ regularity on $\mathbb{S}^{N-1} \times [0,1[$, and it admits a continuous extension
to $\mathbb{S}^{N-1} \times [0,1]$;
\item\label{en:2} $\langle \frac{\partial y_\epsilon}{\partial \tau},y\rangle  > 0$ for any $P\in \mathbb{S}^{N-1}$, $\tau \in\left[0,1\right[$;
\item\label{en:1bis} $y_\epsilon(\cdot,1): \mathbb{S}^{N-1} \rightarrow {\mathcal W}_{\epsilon}^{-1}(E)$ is a homeomorphism;
\item\label{en:4} the potential
\begin{equation}\label{eq:matching}
V_\epsilon(x)=
\left\{
\begin{aligned}
& W_\epsilon(x) \text{ if } \Vert x \Vert > 1,\\
& U_\epsilon(x) \text{ if } \Vert x \Vert \leq 1,
\end{aligned}
\right.
\end{equation}
where $W_\epsilon = {\mathcal W}_{\epsilon} \circ y_{\epsilon}^{-1}$, is of class $C^2$ in an open neighborhood of $V_\epsilon^{-1}\big(\left]-\infty,E\right]\big)$;
\item\label{en:6} $V_\epsilon^{-1}\big(\left]-\infty,E\right]\big)$ is homeomorphic to $B[0;1]$;
\item\label{en:7} $\mathrm{grad }V_\epsilon(x) \neq 0$ for any $x \in V_\epsilon^{-1}(E)$;
\item\label{en:5} as $\epsilon \rightarrow 0$, the potential $V_\epsilon$ converges to a radial potential $V_0$ with respect to the $C^2$--topology in $\bar{\mathcal O}$; $B[0;1] \subset V_0^{-1}\big(\left]-\infty,E\right[\big)$, $V_0^{-1}\big(\left]-\infty,E\right]\big) \subset \mathcal O$, and $E$ is a regular value for $V_0$.
\end{enumerate}
\end{prop}

\begin{proof}
In order to study the global properties of the solution $y_\epsilon$ it is necessary to use some specific properties of the \emph{generating function} $F_\epsilon$
defined in \eqref{F-epsilon}. Initially, it must be observed that the maps $f_{i,\epsilon}$ are uniquely determined in such a way to obtain a $C^2$--matching on $\mathbb{S}^{N-1}$ between $U_\epsilon$ and $W_\epsilon$. Indeed, using $(P,\tau)$ as coordinates (recall Proposition \ref{costruzione-geo}), we have that the potential ${\mathcal W}_\epsilon=W_\epsilon\circ y_\epsilon$  is given by \eqref{W-epsilon}, while $U_\epsilon\circ y_\epsilon$ is given by the map
\[
(P,\tau) \mapsto \frac12\Vert y_\epsilon(P,\tau) - \epsilon v\Vert^2.
\]
Straightforward technical computations show that $C^0$, $C^1$ and $C^2$--regularity are equivalent to the relations
\begin{equation}\label{eq:C0}
f_{1,\epsilon}(P)= \frac12\sqrt{2E-1+2\epsilon \langle P,v\rangle  - \epsilon^2},
\end{equation}
\begin{equation}\label{eq:C1}
f_{2,\epsilon}(P)= \frac14\Big(\frac43 f_{1,\epsilon}^2 - \frac{1}{f_{1,\epsilon}}(1-\epsilon\langle P,v\rangle )\Big),
\end{equation}
\begin{multline}\label{eq:C2}
f_{3,\epsilon}(P)=
\frac 19f_{1,\epsilon}^3
+f_{2,\epsilon}f_{1,\epsilon}
-\frac{2}{f_{1,\epsilon}}\left(\Vert \text{grad }f_{1,\epsilon} \Vert ^2
-\langle \text{grad }f_{1,\epsilon},P\rangle ^2\right)+\\
-\frac 1{12}-\frac 1{4f_{1,\epsilon}}+\frac{f_{2,\epsilon}}{4f_{1,\epsilon}^2}+ \\
\frac{\epsilon}{4}\left[ \frac{\langle\text{grad }f_{1,\epsilon},v\rangle}{f_{1,\epsilon}^2}+\langle P,v\rangle\left(
\frac 13-\frac{f_{2,\epsilon}+{\langle\text{grad }f_{1,\epsilon},P\rangle}}{f_{1,\epsilon}^2}
\right)\right].
\end{multline}
respectively, which are well defined for any $\epsilon$ sufficiently small. Note that,

\begin{multline}\label{limitefiepsilon}
\text{For any } i=1,2,3, \; \lim_{\epsilon \rightarrow 0^+}=f_i, \text{ where the $f_i$'s are the constant maps }\\
f_{1}= \frac12\sqrt{2E-1},
f_{2}= \frac14\Big(\frac43 f_{1}^2 - \frac{1}{f_{1}}\Big),
f_{3}=\frac 19f_{1}^3
+f_{2}f_{1}-\frac 1{12}-\frac 1{4f_{1}}+\frac{f_{2}}{4f_{1}^2},\\
\text{ and the limit is meant in the $C^k$--topology, for all $k\geq 0$}.
\end{multline}

\medskip

Since $E \geq 4$, we have in particular
\begin{equation}\label{fiproperties}
f_1 > 1, f_2 > 0, f_3 > 0.
\end{equation}
Now consider the annulus
\[
B_*=\big\{x \in \mathds{R}^{N}: 1 \leq \Vert x \Vert \leq 3\big\}.
\]
By \eqref{limitefiepsilon}, as $\epsilon\to0$ the map
$F_\epsilon$ converges to the map $F_0$ in the $C^1$--topology on $B_*$, where

\begin{equation}\label{eq:effe0}
F_{0}(x)= \sum_{i=1}^3\frac1{i!}f_{i}\big(\Vert x\Vert-1\big)^i.
\end{equation}
By \eqref{fiproperties},
\begin{equation*}
\Vert x \Vert \geq 1 \Rightarrow \langle \text{grad }F_0(x),x\rangle  > 1,
\end{equation*}
\begin{equation*}
\Vert x \Vert \geq 1 \Rightarrow \Vert \text{grad }F_0(x) \Vert > 1.
\end{equation*}
Then, by \eqref{limitefiepsilon}, for $\epsilon$ sufficiently small,
\begin{equation}\label{eq:monotonia-epsilon}
\langle \text{grad }F_\epsilon(x),x\rangle  \geq 1 \text{ for any }x \in B_*,
\end{equation}
\begin{equation}\label{eq:normagrad-epsilon}
\Vert \text{grad }F_\epsilon(x) \Vert \geq 1, \text{ for any }x \in B_*.
\end{equation}

Now fix $P \in \mathbb{S}^{N-1}$ and denote by $\left[0,\tau(P)\right[$  the maximal interval where $y_\epsilon(P,\tau)$ is defined.
For any $\tau \in \left[0,\tau(P)\right[$ we have
\[
\big\Vert y_\epsilon(P,\tau)-y_\epsilon(P,0)\big\Vert \leq \int_0^\tau\left\Vert \frac{\partial y_\epsilon}{\partial \tau} \right\Vert \,\mathrm ds \leq \int_0^\tau \frac{\mathrm ds}{(1-s)^{\frac13}}
\leq \int_0^1 \frac{\mathrm ds}{(1-s)^{\frac13}} = \frac32.
\]
Then:
\[
\Vert y_\epsilon(P,\tau) \Vert \leq \tfrac52,\quad \text{ for all }\tau \in \left[0,\tau(P)\right[,
\]
and using \eqref{eq:FCauchy2},  \eqref{eq:monotonia-epsilon} and \eqref{eq:normagrad-epsilon}, we obtain \eqref{en:1} and \eqref{en:2}.
Note that the continuous extendibility to $\tau=1$ comes from integrability of ${(1-\tau)^{-\frac13}}$ on $[0,1]$).

Now consider the reparameterization $\tilde y_\epsilon(P,s)$ of $y_\epsilon(P,\tau)$ defined by
\begin{equation}\label{eq:ripyeps}
\tilde y_\epsilon(P,s) = y_\epsilon(P,\psi(s)), \; \tau=\psi(s)=1-\left(1-\frac23s\right)^{\frac32}.
\end{equation}

A simple computation shows that $\tilde y_\epsilon$ solves the Cauchy Problem
\begin{equation}\label{eq:FCauchy3}
\left\{
\begin{aligned}
&\frac{\partial y}{\partial \tau}= \frac{\text{ grad }F_\epsilon(y)}{\Vert \text{ grad }F_\epsilon(y) \Vert^2}\\
& y(0)=P \in \mathbb{S}^{N-1}.
\end{aligned}
\right.
\end{equation}
Proposition \ref{costruzione-geo}, applied with $\phi\equiv 1$, proves part~\eqref{en:1bis}, and shows that $\tilde y_\epsilon$ is a $C^\infty$ diffeomorphism between a neighborhood of $\mathbb{S}^{N-1} \times [0,\frac32]$
and a neighborhood of $F_\epsilon^{-1}\big([0,\frac32]\big)$. This yields the $C^\infty$--regularity of $W_\epsilon$ in a neighborhood of $W_\epsilon^{-1}(E)$, since,
using coordinates $(P,s)$, the potential $W_\epsilon$ is described by
\begin{equation}\label{eq:vus}
\mathcal W_\epsilon(P,s)=E-2\langle \text{grad }F_\epsilon\big(\tilde y_\epsilon(P,s)\big),\text{grad }F_\epsilon\big(\tilde y_\epsilon(P,s)\big)\rangle\big(1-\tfrac23s\big),
\end{equation}
while
\[
\mathcal W_\epsilon(P,s) = E \quad \Longleftrightarrow\quad s=\frac32.
\]
Then, using the $C^2$--regularity of $V_\epsilon$ in a neighborhood of $\mathbb{S}^{N-1}$, we deduce part~\eqref{en:4}.

In order to prove part~\eqref{en:6}, observe that
\[
\big\{x: V_{\epsilon}(x) \leq E\big\}= B[0;1] \cup B_0,
\]
where $B_0=\big\{\tilde y_{\epsilon}(P,\tau): P \in \mathbb{S}^{N-1},\ s \in\left[0,\tfrac32\right]\big\}$.
The conclusion follows since $B[0;1] \cap B_0 = \mathbb{S}^{N-1}$, $\tilde y_{\epsilon}(P,0)=P$ for every $P \in \mathbb{S}^{N-1}$ and $B_0$
is homeomorphic to $\mathbb{S}^{N-1} \times\big[0,\frac32\big]$.

Finally, part~\eqref{en:7} comes from the property
\[
\mathcal W_\epsilon = E\quad \Longrightarrow\quad \frac{\partial \mathcal W_\epsilon}{\partial s}= \tfrac13 \big\Vert \text{ grad }F_\epsilon\left(\tilde y_\epsilon(P,3/2)\right)\big\Vert^2 \not=0,
\]
while the existence of $\mathcal O$ satisfying part~\eqref{en:5} is a consequence of property \eqref{limitefiepsilon}.
\end{proof}

In next two Lemmas we will establish the correspondence between Orthogonal Geodesic Chords and Jacobi geodesics arriving on the boundary of the potential well.

\begin{lem}\label{lem:passo1}
Let $\epsilon_*$ be as in Proposition \ref{th:F+V}, and let $\epsilon \in\left]0,\epsilon_*\right]$ be fixed.
There exists a continuous map $\mathbb S^{N-1}\ni P\mapsto T_{\epsilon}(P) > 0$ such that the solution $q_\epsilon(P,t)$ of the
Cauchy problem with total energy $E$
\begin{equation}\label{eq:startingesterne}
\left\{
\begin{aligned}
& \ddot q + \text{grad }V_\epsilon(q) = 0 \\
& q(0) = P \\
& \dot q(0)=\lambda_{\epsilon}P,
\end{aligned}
\right.
\end{equation}
where \[\lambda_{\epsilon}=\sqrt{2E-1+2\epsilon\langle P,v\rangle -\epsilon^2},\] reaches $V_\epsilon^{-1}(E)$ at the time $T_{\epsilon}(P)$
and $V_\epsilon(P,q_\epsilon(t)) < E$ for any $t \in\left[0,T_\epsilon(P)\right[$. Moreover, $t\mapsto\Vert q_\epsilon(P,t)\Vert $ is strictly increasing
for all $P$.
\end{lem}
\begin{proof}
Let $y_\epsilon$ be the solution of the Cauchy problem \eqref{eq:FCauchy2}. Using Proposition~\ref{thm:recGaussLemma}
and Proposition~\ref{costruzione-geo}, for any fixed $P \in \mathbb{S}^{N-1}$ the curve $\tau \mapsto y_\epsilon(P,\tau)$ ($\tau \in [0,1[$) is a unit speed geodesic with respect to the Jacobi metric \eqref{eq:Jacobi} where $W=W_\epsilon$.

Now by  Maupertuis Principle (cf. Lemma \ref{cambio-variabile}) the curve
\[
q_\epsilon(P,t)=y_\epsilon(P,\sigma(P,t)),
\]
with $\sigma$ solution of \eqref{eq:sigmacauchyproblem}, is a solution of \eqref{eq:startingesterne} having total energy $E$.

Now note that $q_\epsilon$ reaches $V_\epsilon^{-1}(E)$ when $\sigma(P,t)$ reaches the value $1$, so we have to prove that there exists a continuous map
$P\mapsto T_\epsilon(P)$ such that
\begin{equation}\label{eq:tep}
\sigma\big(P,T_\epsilon(P)\big)=1, \; \sigma(P,t) < 1 \text{ for any }t \in\left[0,T_\epsilon(P)\right[.
\end{equation}
Now, by \eqref{eq:sigmacauchyproblem},
\begin{equation*}
\left\{
\begin{aligned}
& \dot \sigma = E-\mathcal W_\epsilon(y_\epsilon(P,\sigma(t)) = 2(1-\sigma(t))^{\frac23}\big\Vert \text{ grad }F_\epsilon(y_\epsilon(P,\sigma(t)))\big\Vert^2 \\
& \sigma(0) = 0.
\end{aligned}
\right.
\end{equation*}
Therefore
\[
\int_0^t \frac{\dot \sigma(s)}{(1-\sigma(s))^{\frac23}}\,\mathrm ds = 2\int_0^t\big\Vert \text{ grad }F_\epsilon(P,\sigma(P,s))\big\Vert^2 \,\mathrm ds
\]
and, since $\sigma(0)=0$, $T_\epsilon(P)$ is uniquely defined by
\[
2\int_{0}^{T_\epsilon(P)}\big\Vert \text{ grad }F_\epsilon(y_\epsilon(P,\sigma(s)))\big\Vert^2 \,\mathrm ds = \int_0^1 \frac{\,\mathrm ds}{(1-s)^{\frac23}} = 3,
\]
from which we deduce \eqref{eq:tep}.
Finally, by \eqref{en:2} of Proposition \ref{th:F+V} we have immediately that $t\mapsto\big\Vert  q_\epsilon(t) \big\Vert$ is strictly increasing
and the proof is complete.
\end{proof}

\begin{lem}\label{lem:passo2}
Let $T_\epsilon(P)$ be as in Lemma~\ref{lem:passo1}.
For any $x\in V_\epsilon^{-1}(E)$ there exists a unique $P \in \mathbb{S}^{N-1}$ such that the solution $q(x,t)$ of \eqref{eq:startingBO} with $q(0)=x$
satisfies
\begin{equation}\label{eq:gammaPx}
q(x,t)=q_\epsilon\big(P,T_\epsilon(P)-t)\big),
\end{equation}
where  $q_\epsilon(P,t)$ is
the solution of \eqref{eq:startingesterne} in Lemma~\ref{lem:passo1}.
\end{lem}
\begin{proof}
By \eqref{en:1bis} of Proposition~\ref{th:F+V}, for any $x \in V_\epsilon^{-1}(E)$, there exists a unique $P\in \mathbb{S}^{N-1}$ such that $y_\epsilon(P,1) = x$. This means that there exists a unique $P \in \mathbb{S}^{N-1}$ such that $q_\epsilon\big(P,T_\epsilon(P)\big)=x$, where  $q_\epsilon(P,t)$ is
the solution of \eqref{eq:startingesterne} in Lemma~\ref{lem:passo1}. Then, the local uniqueness of solutions for the Cauchy problem \eqref{eq:startingBO} gives  \eqref{eq:gammaPx}.
\end{proof}

\section{Examples with only two brake orbits and homoclinics}\label{sec:examples}
Fix $E \geq 4$ and $\epsilon_* \in ]0,1[$ as in Proposition \ref{th:F+V}. Let $\psi : \mathds{R}^+ \rightarrow \mathds{R}^+$ of class $C^2$ with the following properties:
\begin{itemize}
\item $\exists s_*\in]0,\tfrac18[$ such that $\psi(s_*)= E$;  
\item $\exists\delta_*\in]0,\tfrac12-s_*[$ such that $\psi(s_*)<E,\,\forall s\in]s_*,\tfrac12-\delta_*[$;
\item $\psi(s)=s,\,\forall s\,:\,|s-\tfrac12|\le\delta_*$.
\end{itemize}
\medskip
Consider the potential
\[
\tilde{U_\epsilon}(x)=\psi(\frac12\Vert x-\epsilon v\Vert^2),
\]
and set
\begin{equation}\label{eq:matching2}
\tilde{V_\epsilon}(x)=
\left\{
\begin{aligned}
& W_\epsilon(x) \text{ if } \Vert x \Vert > 1,\\
& \tilde{U_\epsilon}(x) \text{ if } \Vert x \Vert \leq 1,
\end{aligned}
\right.
\end{equation}
where $W_\epsilon$ is the potential $\mathcal W_\epsilon$ of
\eqref{W-epsilon} written in cartesian coordinates.

Observe that there exists $\epsilon_1 \in [0,\epsilon_*]$ such that, for any $\epsilon \in ]0,\epsilon_1]$, $\tilde V_\epsilon$ is a $C^2$--potential for which the potential well $\tilde V_\epsilon^{-1}(]-\infty,E])$ is homeomorphic to the $N$-dimensional annulus, and $\tilde V_\epsilon^{-1}(E)$ is a $C^2$--hypersurface consisting of two connected components each of them homeomorphic to $\mathbb{S}^{N-1}$.

\begin{lem}\label{rem:preparatoria}
Let $\epsilon_*$ be chosen as in Proposition~\ref{th:F+V}.
There exists $\bar \epsilon \in]0, \epsilon_*]$ such that for any $\epsilon \in\left]0,\bar \epsilon\right]$
 and for any solution of the Cauchy problem
\begin{equation}\label{eq:ortU-epsilon}
\left\{
\begin{aligned}
& \ddot q + \text{grad }\tilde U_\epsilon(q) = 0 \\
& q(0) = P \\
& \dot q(0)=-\lambda_{\epsilon}P
\end{aligned}
\right.
\end{equation}
with $\lambda_\epsilon= \sqrt{2E-1+2\epsilon\langle P,v\rangle -\epsilon^2}$, there exists a continuous map $\mathbb S^{N-1}\ni P\mapsto t_\epsilon(P) > 0$ such that
\begin{multline*}
q\big(\left]0,t_\epsilon(P)\right[\big) \subset \{\Vert x-\epsilon v\Vert >\frac12, \;\Vert x \Vert < 1\},
\big\Vert q\big(t_\epsilon(P)\big)-\epsilon v\big\Vert =\frac12,\quad\text{and} \\
\left\langle \dot q\big(t_\epsilon(P)\big), q\big(t_\epsilon(P)\big)\right\rangle  <0, \;
\left\langle \dot q\big(t_\epsilon(P)\big), q\big(t_\epsilon(P)\big)-\epsilon v\right\rangle  <0.
\end{multline*}
\end{lem}
\begin{proof}
It is a simple consequence of the $C^2$-convergence of $\tilde U_\epsilon$ to the radial potential $\tilde U_0$ on the unit disk.
\end{proof}

Set $\mathcal D_1 = \{\tfrac12\Vert x-\epsilon v\Vert^2=s_*\}$ and denote by $\mathcal D_2$ the other connected component of $\tilde{V_\epsilon}^{-1}(E)$.
Using again the $C^2$--convergence to a radial potential, the following result can be also established.
\begin{lem}\label{lem:CC}
 There exists $\epsilon_0 \in\left ]0,\bar \epsilon\right]$ such that for any $\epsilon \in\left]0,\epsilon_0\right]$, any brake orbit
 connects $\mathcal D_1$ with $\mathcal D_2$.
\end{lem}

We are now ready to prove the following theorem that gives examples of potential wells homeomorphic to the annulus
having only two brake orbits.

\begin{teo}\label{thm:deuBO} Fix $E \geq 4 $ as in Proposition \ref{th:F+V}. There exists $\tilde \epsilon \in\left]0,\epsilon_0\right]$ such that, for any
$\epsilon\in\left]0,\tilde\epsilon\right]$, there exist
only two brake orbits in the potential well $\tilde V_\epsilon^{-1}\big(\left]-\infty,E\right]\big)$ connecting the two different components of
$\tilde{V_\epsilon}^{-1}(E)$.
\end{teo}

\begin{proof}
Let $q$ be a brake orbit of energy $E$. By Lemma \ref{lem:CC} we can assume that $q(0) \in \mathcal D_2$.
By Lemmas~\ref{lem:passo1} and \ref{lem:passo2},
there exists $P_0\in\mathbb S^{N-1}$ which is the first intersection point of $q$ with $\mathbb S^{N-1}$, occurring at the time $T_\epsilon(P_0)>0$,
and $\dot q\big(T_\epsilon(P_0)\big)$ is orthogonal to $\mathbb S^{N-1}$ at $P_0$.

Note that, by the monotonicity property in Lemma \ref{lem:passo1},
\begin{equation}\label{eq:arrivo-partenza}
\Vert q(t) \Vert > 1 \text{ for every }t\in\left[0,T_\epsilon(P_0)\right[.
\end{equation}

\begin{figure}
\includegraphics[scale=.7]{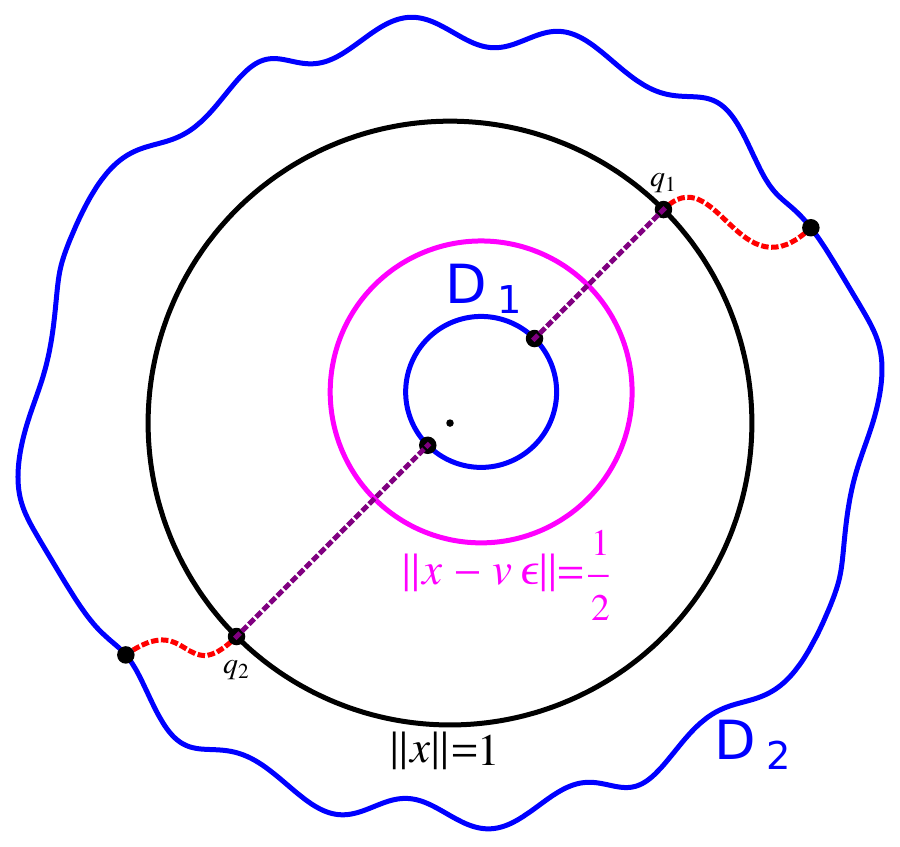}
\caption{Behavior of the brake orbits with energy $E$ for the potential $\tilde{V_\epsilon}$ \eqref{eq:matching2}, whenever $\epsilon$ is sufficiently small. $\mathcal D_1$ and $\mathcal D_2$ (in blue) are the connected components of the $E$--level set of $\tilde{V_\epsilon}$.}
\label{fig:3}
\end{figure}
Now, by \eqref{eq:arrivo-partenza} and Lemma \ref{rem:preparatoria}, at the instant $T_\epsilon(P_0)$ the curve $q$ moves (orthogonally) inside the ball $\{x:\Vert x \Vert < 1\}$ and it arrives on $\mathcal D_1$ with null speed at the instant $t_\epsilon(q(T_\epsilon(P_0))$.

Then, considering \cite[Example 1]{arma} and using Lemmas \ref{lem:passo1} and \ref{lem:passo2} we see that there are only two points $q_1$ and $q_2$ as $P_0$ above (see Figure~\ref{fig:3}), and this allows to get the conclusion of the proof.
\end{proof}

Note that if we choose $\psi$ with the additional property that $s=s_*$ is a non degenerate maximum point, by the same proof of the above Theorem we also get the following result.
\begin{teo}\label{thm:deuHOM} Fix $E \geq 4 $ as in Proposition \ref{th:F+V}. There exists $\tilde \epsilon \in\left]0,\epsilon_0\right]$ such that, for any
$\epsilon\in\left]0,\tilde\epsilon\right]$, there exist
only two homoclinics in the potential well $\tilde V_\epsilon^{-1}\big(\left]-\infty,E\right]\big)$ emanating from the nondegenerate maximum point and reaching the regular maximal hypersurface included in $\tilde V_\epsilon^{-1}(E)$.
\end{teo}

\end{document}